\documentclass[preprint,12pt]{elsarticle}

\usepackage{amsmath,amssymb,amsthm,mathtools}
\usepackage{enumitem}
\usepackage{geometry}
\usepackage[hidelinks]{hyperref}
\usepackage[utf8]{inputenc}
\numberwithin{equation}{section}

\newtheorem{theorem}{Theorem}[section]
\newtheorem{lemma}[theorem]{Lemma}
\newtheorem{proposition}[theorem]{Proposition}
\newtheorem{corollary}[theorem]{Corollary}
\newtheorem{definition}[theorem]{Definition}
\newtheorem{remark}[theorem]{Remark}
\newtheorem{example}[theorem]{Example}

\newcommand{\E}{\mathcal{E}}

\newcommand{\cl}{\operatorname{cl}}
\newcommand{\Fix}{\operatorname{Fix}}
\newcommand{\diam}{\operatorname{diam}}

\begin{document}

\begin{frontmatter}

\title{Lifted Takahashi Convexity on the Isbell-Convex Hull of an Asymmetrically Normed Real Vector Space}

\author{Philani Rodney Majozi}
\ead{Philani.Majozi@nwu.ac.za}

\author{Mcedisi Sphiwe Zweni}
\ead{Sphiwe.Zweni@nwu.ac.za}

\address{School of Mathematical and Statistical Sciences, Pure and Applied Analytics, North-West University, Mahikeng Campus, South Africa}

\begin{abstract}
K\"unzi and Yildiz introduced convexity structures in the sense of Takahashi
for \(T_{0}\)-quasi-metric spaces. In this article, we continue this line of
study on the Isbell-convex hull of an asymmetrically normed real vector space.
Using the canonical hull quasi-metric and the vector-space operations on
\(\E(X,\|\cdot\|)\), we define a lifted convexity structure
\[
\mathbb W(f,g,\lambda)=\lambda f\oplus(1-\lambda)g
\]
and show that \((\E(X,\|\cdot\|),q_{\E},\mathbb W)\) is a convex
\(T_{0}\)-quasi-metric space. We further prove compatibility with the canonical
embedding and relate the construction to \(W\)-convexity of minimal function
pairs.
\end{abstract}

\begin{keyword}
Takahashi convexity \sep $T_{0}$-quasi-metric space \sep Isbell-convex hull
\sep asymmetrically normed space \sep minimal function pair \sep nonexpansive mapping
\\[2pt]
\end{keyword}

\end{frontmatter}

\section{Introduction}

Convexity structures in the sense of Takahashi provide a flexible way to
transfer classical convex-geometric arguments to settings where no linear
structure is available.  In his seminal work, Takahashi introduced a convexity
map on a metric space \((X,d)\) as a rule \(W(x,y,\lambda)\) producing
``metric convex combinations'' and satisfying the characteristic inequality
\[
d\bigl(z,W(x,y,\lambda)\bigr)
\le
\lambda d(z,x)+(1-\lambda)d(z,y),
\qquad z\in X,\ \lambda\in[0,1].
\]
This framework leads naturally to convex subsets, metric segments and fixed
point principles for nonexpansive mappings in convex metric spaces
\cite{Takahashi1970}.  Later developments refined this theory by considering
additional rigidity properties, such as uniqueness and higher-dimensional
barycentric operations, and by connecting convexity structures with compactness
and fixed point theorems for condensing multifunctions \cite{Talman1977}.

A systematic asymmetric counterpart was developed by K\"unzi and Yildiz, who
extended Takahashi's theory to \(T_{0}\)-quasi-metric spaces.  In this setting
one must require two directed inequalities: one for the quasi-metric \(q\) and
one for its conjugate \(q^{t}\).  This two-sided requirement is essential in
the asymmetric case and leads to the notion of a convex
\(T_{0}\)-quasi-metric space \((X,q,W)\).  Their theory includes permanence
properties, double closure for convex sets, uniqueness of convexity structures,
and a segment theory in which unique convexity structures yield isometric
parametrisations by directed intervals \cite{KunziYildiz2016}.

Injective-envelope ideas from metric geometry also admit an asymmetric
analogue through the Isbell-convex hull.  For di-spaces, that is,
\(T_{0}\)-quasi-metric spaces, Kemajou, K\"unzi and Olela-Otafudu introduced
Isbell-convexity as the appropriate replacement of hyperconvexity and proved
that a di-space is Isbell-convex if and only if it is di-injective
\cite{KemajouKunziOlela2012}.  They also gave an explicit construction of the
Isbell-hull by means of extremal ample function pairs.  This hull is uniquely
determined up to isometry and contains the original di-space isometrically by
the canonical embedding \(x\mapsto f_x\).

When the di-space arises from an asymmetrically normed real vector space
\((X,\|\cdot\|)\), the corresponding hull will be denoted by
\(\E(X,\|\cdot\|)\).  A distinctive feature of this asymmetric linear setting
is that the Isbell-hull is not only a directed metric enlargement.  Conradie,
K\"unzi and Olela-Otafudu showed that \(\E(X,\|\cdot\|)\) admits scalar
multiplication and an addition operation \(\oplus\), turning it into a real
vector space; moreover, under the specialization order induced by the hull
quasi-metric, it becomes a Dedekind complete vector lattice
\cite{ConradieKunziOlela2017}.  Related work of Olela-Otafudu on extremal
function pairs further supports the view of the hull as a natural
convex-analytic object associated with \((X,\|\cdot\|)\) \cite{Olela2014}.

The present paper is a direct continuation of \cite{OlelaZweni2023}, where
\(W\)-convexity for real-valued function pairs on a convex asymmetrically
normed space \((X,\|\cdot\|,W)\) was studied.  In that work it was shown that,
under suitable translation-invariance and homogeneity assumptions on \(W\),
minimal pairs in \(\E(X,\|\cdot\|)\) are \(W\)-convex when regarded as
function pairs on the base space \(X\).  Thus the predecessor paper established
a functional convexity principle: hull elements satisfy Jensen-type
inequalities on \(X\).

The present paper takes the next step.  Instead of studying only convexity of
hull elements as functions on \(X\), we place a convexity structure directly
on the hull itself.  The guiding question is the following.

\begin{quote}
\emph{Can the Isbell-convex hull \(\E(X,\|\cdot\|)\) itself be endowed with a
Takahashi convexity structure that is compatible with the canonical embedding
\(i:X\to\E(X,\|\cdot\|)\) and with the algebraic operations on the hull?}
\end{quote}

We answer this question affirmatively.  Using the canonical Isbell-hull
di-metric \(q_{\E}\) and the vector-space operations on
\(\E(X,\|\cdot\|)\), we define the lifted barycentric map
\[
\mathbb W:\E(X,\|\cdot\|)\times \E(X,\|\cdot\|)
\times[0,1]\longrightarrow \E(X,\|\cdot\|),
\qquad
\mathbb W(f,g,\lambda)=\lambda f\oplus(1-\lambda)g.
\]
We prove that
\[
(\E(X,\|\cdot\|),q_{\E},\mathbb W)
\]
is a convex \(T_{0}\)-quasi-metric space in the sense of
Takahashi--K\"unzi--Yildiz.  Thus the hull carries an intrinsic directed
convex geometry, not merely convexity inequalities for its elements viewed as
functions on \(X\).

We also prove compatibility with the original space.  If
\[
S(x,y,\lambda)=\lambda x+(1-\lambda)y
\]
is the standard affine convexity structure on \(X\), then the canonical
embedding \(i:X\to\E(X,\|\cdot\|)\) satisfies
\[
i\bigl(S(x,y,\lambda)\bigr)
=
\mathbb W\bigl(i(x),i(y),\lambda\bigr),
\qquad x,y\in X,\ \lambda\in[0,1].
\]
Consequently \(i(X)\) is \(\mathbb W\)-convex in the hull.

After constructing the lifted convexity structure, we relate it to the earlier
functional \(W\)-convexity of minimal pairs.  We show that, under natural
translation-invariance and homogeneity assumptions, \(W\)-convexity is stable
under the hull operations used to define \(\mathbb W\).  We then study hull
segments and show that, when the lifted convexity structure is unique, the
segment between two hull elements is isometrically parametrised by a standard
directed interval.  Examples and counterexamples illustrate the role of the
directed quasi-metric structure and show that convexity for the symmetrised
metric does not in general imply convexity for the underlying
\(T_{0}\)-quasi-metric.

\medskip

\noindent\textbf{Organisation.}
Section~2 recalls the necessary background on \(T_{0}\)-quasi-metrics,
Takahashi convexity structures, asymmetrically normed spaces, Isbell-convex
hulls and the algebraic operations on the hull.  Section~3 constructs the
lifted Takahashi convexity structure on \(\E(X,\|\cdot\|)\) and proves
compatibility with the canonical embedding.  Section~4 relates intrinsic hull
convexity to \(W\)-convexity of minimal function pairs.  Section~5 studies
segments and uniqueness of the lifted structure.  Section~6 gives examples,
counterexamples and concluding directions.

\section{Preliminaries}\label{sec:prelim}

In this section we recall the basic notions used throughout the paper.  We
first fix the notation for \(T_{0}\)-quasi-metric spaces and Takahashi
convexity structures.  We then recall asymmetrically normed real vector spaces,
the associated \(T_{0}\)-quasi-metric, the Isbell-convex hull, and the
algebraic operations on the hull.

\subsection{\texorpdfstring{\(T_{0}\)}{T0}-quasi-metrics and directed notation}
\label{subsec:qmetric}

\begin{definition}\label{def:qpm}
Let \(X\) be a set.  A map
\[
q:X\times X\longrightarrow [0,\infty)
\]
is called a \emph{quasi-pseudometric} if
\[
q(x,x)=0\qquad (x\in X),
\]
and
\[
q(x,z)\le q(x,y)+q(y,z)
\qquad (x,y,z\in X).
\]
It is called a \emph{\(T_{0}\)-quasi-metric} if, in addition,
\[
q(x,y)=0=q(y,x)\quad\Longrightarrow\quad x=y.
\]
A \(T_{0}\)-quasi-metric space will also be called a \emph{di-space}.
\end{definition}

Given a quasi-pseudometric \(q\), its conjugate is
\[
q^{t}(x,y):=q(y,x),
\]
and its symmetrisation is
\[
q^{s}:=q\vee q^{t},\qquad
q^{s}(x,y)=\max\{q(x,y),q(y,x)\}.
\]
If \(q\) is \(T_{0}\), then \(q^{s}\) is a metric.

For \(a,b\in\mathbb R\), we write
\[
a\dot{-}b:=\max\{a-b,0\}.
\]
The standard \(T_{0}\)-quasi-metric on \(\mathbb R\) is
\[
u(a,b)=a\dot{-}b.
\]

For \(x\in X\) and \(\varepsilon>0\), the open \(q\)-ball is
\[
B_q(x,\varepsilon)=\{y\in X:q(x,y)<\varepsilon\}.
\]
For \(\varepsilon\ge 0\), the closed \(q\)-ball is
\[
C_q(x,\varepsilon)=\{y\in X:q(x,y)\le \varepsilon\}.
\]
The topology generated by the open \(q\)-balls is denoted by \(\tau(q)\).

\begin{definition}\label{def:doubleclosure}
Let \((X,q)\) be a \(T_{0}\)-quasi-metric space and let \(A\subseteq X\).
The \emph{double closure} of \(A\) is
\[
\cl_{\tau(q)}A\cap \cl_{\tau(q^{t})}A.
\]
We say that \(A\) is \emph{doubly closed} if
\[
A=\cl_{\tau(q)}A\cap \cl_{\tau(q^{t})}A.
\]
\end{definition}

\subsection{Takahashi convexity structures}
\label{subsec:TCS}

\begin{definition}\label{def:TakahashiW}
Let \((X,q)\) be a \(T_{0}\)-quasi-metric space.  A map
\[
W:X\times X\times[0,1]\longrightarrow X
\]
is called a \emph{Takahashi convexity structure}, or briefly a \emph{TCS}, if
for all \(x,y,z\in X\) and all \(\lambda\in[0,1]\),
\[
q\bigl(z,W(x,y,\lambda)\bigr)
\le
\lambda q(z,x)+(1-\lambda)q(z,y),
\]
and
\[
q\bigl(W(x,y,\lambda),z\bigr)
\le
\lambda q(x,z)+(1-\lambda)q(y,z).
\]
In this case \((X,q,W)\) is called a \emph{convex \(T_{0}\)-quasi-metric space}.
\end{definition}

\begin{remark}\label{rem:TCS-dual}
The second inequality in Definition~\ref{def:TakahashiW} is the first
inequality written for the conjugate quasi-metric \(q^{t}\).  Thus a TCS on
\((X,q)\) is simultaneously compatible with both directed distances \(q\) and
\(q^{t}\).
\end{remark}

\begin{proposition}\label{prop:qts}
Let \((X,q,W)\) be a convex \(T_{0}\)-quasi-metric space. Then
\((X,q^{t},W)\) is also a convex \(T_{0}\)-quasi-metric space. Moreover,
\((X,q^{s},W)\) is a convex metric space in the sense of Takahashi.
\end{proposition}

\begin{proof}
The assertion for \(q^{t}\) follows directly from the two defining
inequalities in Definition~\ref{def:TakahashiW}. Indeed, since
\(q^{t}(x,y)=q(y,x)\), the first Takahashi inequality for \(q^{t}\) is exactly
the second Takahashi inequality for \(q\), and the second Takahashi inequality
for \(q^{t}\) is exactly the first Takahashi inequality for \(q\).

For \(q^{s}=q\vee q^{t}\), let \(x,y,z\in X\) and \(\lambda\in[0,1]\). Then
\[
q^{s}\bigl(z,W(x,y,\lambda)\bigr)
=
\max\bigl\{
q\bigl(z,W(x,y,\lambda)\bigr),
q^{t}\bigl(z,W(x,y,\lambda)\bigr)
\bigr\}.
\]
Using the convexity inequalities for \(q\) and \(q^{t}\), we obtain
\[
q^{s}\bigl(z,W(x,y,\lambda)\bigr)
\le
\lambda q^{s}(z,x)+(1-\lambda)q^{s}(z,y).
\]
Similarly,
\[
q^{s}\bigl(W(x,y,\lambda),z\bigr)
\le
\lambda q^{s}(x,z)+(1-\lambda)q^{s}(y,z).
\]
Thus \((X,q^{s},W)\) is a convex metric space.
\end{proof}

\begin{definition}\label{def:Wconvexset}
Let \((X,q,W)\) be a convex \(T_{0}\)-quasi-metric space.  A subset
\(C\subseteq X\) is called \emph{\(W\)-convex} if
\[
W(x,y,\lambda)\in C
\]
whenever \(x,y\in C\) and \(\lambda\in[0,1]\).
\end{definition}

\begin{definition}\label{def:synchronized}
A TCS \(W\) on \((X,q)\) is called \emph{synchronized} if
\[
W(x,y,\lambda)=W(y,x,1-\lambda)
\]
for all \(x,y\in X\) and \(\lambda\in[0,1]\).
\end{definition}

\begin{definition}\label{def:propertyS}
A TCS \(W\) on \((X,q)\) is said to have \emph{property \({\rm(S)}\)} if
\[
q\bigl(W(x,y,\lambda),W(x',y',\lambda)\bigr)
\le
\lambda q(x,x')+(1-\lambda)q(y,y')
\]
for all \(x,y,x',y'\in X\) and all \(\lambda\in[0,1]\).
\end{definition}

\begin{definition}\label{def:uniqueTCS}
Let \(W\) be a TCS on a \(T_{0}\)-quasi-metric space \((X,q)\).  We say that
\(W\) is a \emph{unique TCS} if whenever \(w\in X\) satisfies, for some
\(x,y\in X\) and \(\lambda\in[0,1]\),
\[
q(z,w)\le \lambda q(z,x)+(1-\lambda)q(z,y)
\]
and
\[
q(w,z)\le \lambda q(x,z)+(1-\lambda)q(y,z)
\]
for every \(z\in X\), then
\[
w=W(x,y,\lambda).
\]
\end{definition}

\subsection{Asymmetrically normed real vector spaces}
\label{subsec:asynorm}

\begin{definition}\label{def:asynorm}
Let \(X\) be a real vector space.  A map
\[
\|\cdot\|:X\longrightarrow [0,\infty)
\]
is called an \emph{asymmetric seminorm} if, for all \(x,y\in X\) and all
\(\alpha\ge 0\),
\[
\|\alpha x\|=\alpha\|x\|,
\qquad
\|x+y\|\le \|x\|+\|y\|.
\]
If, in addition,
\[
\|x\|=\|-x\|=0\quad\Longrightarrow\quad x=0,
\]
then \(\|\cdot\|\) is called an \emph{asymmetric norm}, and
\((X,\|\cdot\|)\) is called an \emph{asymmetrically normed real vector space}.
\end{definition}

The conjugate asymmetric norm is
\[
\|x\|^{t}:=\|-x\|,
\]
and the symmetrised norm is
\[
\|x\|^{s}:=\max\{\|x\|,\|-x\|\}.
\]

Throughout this paper, the asymmetric norm \(\|\cdot\|\) induces the
\(T_{0}\)-quasi-metric
\[
q_{\|\cdot\|}(x,y):=\|x-y\|,
\qquad x,y\in X.
\]
Its conjugate is
\[
q_{\|\cdot\|}^{t}(x,y)=q_{\|\cdot\|}(y,x)=\|y-x\|,
\]
and its symmetrisation is
\[
q_{\|\cdot\|}^{s}(x,y)
=
\max\{\|x-y\|,\|y-x\|\}.
\]

\begin{example}\label{ex:standard-affine}
Let \(C\) be a convex subset of an asymmetrically normed real vector space
\((X,\|\cdot\|)\).  Then
\[
S(x,y,\lambda)=\lambda x+(1-\lambda)y,
\qquad x,y\in C,\ \lambda\in[0,1],
\]
defines a synchronized TCS on \((C,q_{\|\cdot\|})\).  Indeed, for
\(x,y,z\in C\),
\[
q_{\|\cdot\|}\bigl(S(x,y,\lambda),z\bigr)
=
\|\lambda(x-z)+(1-\lambda)(y-z)\|
\le
\lambda\|x-z\|+(1-\lambda)\|y-z\|,
\]
and similarly,
\[
q_{\|\cdot\|}\bigl(z,S(x,y,\lambda)\bigr)
\le
\lambda q_{\|\cdot\|}(z,x)+(1-\lambda)q_{\|\cdot\|}(z,y).
\]
\end{example}

\begin{definition}\label{def:TIH}
Let \(W\) be a TCS on a real vector space \(X\).  We say that \(W\) is
\emph{translation-invariant} if
\[
W(x+z,y+z,\lambda)=W(x,y,\lambda)+z
\]
for all \(x,y,z\in X\) and all \(\lambda\in[0,1]\).  We say that \(W\) is
\emph{homogeneous} if
\[
W(\alpha x,\alpha y,\lambda)=\alpha W(x,y,\lambda)
\]
for all \(x,y\in X\), all \(\lambda\in[0,1]\), and all
\(\alpha\in\mathbb R\).
\end{definition}

\subsection{The Isbell-convex hull}
\label{subsec:isbell}

Let \((X,q)\) be a \(T_{0}\)-quasi-metric space.  A pair
\[
f=(f_1,f_2),\qquad f_i:X\to[0,\infty),
\]
is called an \emph{ample pair} if
\[
q(x,y)\le f_2(x)+f_1(y)
\qquad (x,y\in X).
\]
An ample pair \(f\) is called \emph{minimal} if whenever \(g=(g_1,g_2)\) is
ample and
\[
g_1\le f_1,\qquad g_2\le f_2
\]
pointwise, then \(g=f\).

\begin{definition}\label{def:IsbellHull}
The \emph{Isbell-convex hull} of \((X,q)\) is the set
\[
\mathcal E(X,q)
=
\{\,f=(f_1,f_2): f \text{ is a minimal ample pair on }(X,q)\,\}.
\]
For an asymmetrically normed real vector space \((X,\|\cdot\|)\), we write
\[
\E(X,\|\cdot\|)
:=
\mathcal E\bigl(X,q_{\|\cdot\|}\bigr).
\]
\end{definition}

For \(a\in X\), define
\[
f_a=\bigl((f_a)_1,(f_a)_2\bigr)
\]
by
\[
(f_a)_1(x)=q_{\|\cdot\|}(a,x)=\|a-x\|,
\qquad
(f_a)_2(x)=q_{\|\cdot\|}(x,a)=\|x-a\|.
\]
The map
\[
i:X\longrightarrow \E(X,\|\cdot\|),
\qquad
i(a)=f_a,
\]
is called the \emph{canonical embedding}.

\begin{lemma}\label{lem:suprep}
Let \(f=(f_1,f_2)\in\E(X,\|\cdot\|)\).  Then for every \(x\in X\),
\[
f_1(x)
=
\sup_{y\in X}
\bigl(q_{\|\cdot\|}(y,x)\dot{-}f_2(y)\bigr)
=
\sup_{y\in X}
\bigl(\|y-x\|\dot{-}f_2(y)\bigr),
\]
and
\[
f_2(x)
=
\sup_{y\in X}
\bigl(q_{\|\cdot\|}(x,y)\dot{-}f_1(y)\bigr)
=
\sup_{y\in X}
\bigl(\|x-y\|\dot{-}f_1(y)\bigr).
\]
\end{lemma}

\begin{remark}\label{rem:IsbellD}
On \(\E(X,\|\cdot\|)\), the canonical Isbell-hull di-metric is
\[
D(f,g)
=
\sup_{x\in X}\bigl(f_1(x)\dot{-}g_1(x)\bigr)
\vee
\sup_{x\in X}\bigl(g_2(x)\dot{-}f_2(x)\bigr).
\]
For minimal pairs, this reduces to
\[
D(f,g)
=
\sup_{x\in X}\bigl(f_1(x)\dot{-}g_1(x)\bigr)
=
\sup_{x\in X}\bigl(g_2(x)\dot{-}f_2(x)\bigr).
\]
In Section~\ref{sec:hullconvex}, this di-metric will be denoted by \(q_{\E}\).
\end{remark}

\subsection{Algebraic operations on the Isbell-convex hull}
\label{subsec:ops}

For an asymmetrically normed real vector space \((X,\|\cdot\|)\), the
Isbell-convex hull \(\E(X,\|\cdot\|)\) carries canonical scalar multiplication
and addition operations.  These operations make the hull into a real vector
space and make the canonical embedding \(i:X\to\E(X,\|\cdot\|)\) linear.

\begin{definition}\label{def:scalarHull}
Let \(f=(f_1,f_2)\in\E(X,\|\cdot\|)\) and let \(\alpha\in\mathbb R\).  Define
\(\alpha f=((\alpha f)_1,(\alpha f)_2)\) as follows.

If \(\alpha>0\), then
\[
(\alpha f)_1(x)=\alpha f_1(\alpha^{-1}x),
\qquad
(\alpha f)_2(x)=\alpha f_2(\alpha^{-1}x).
\]
If \(\alpha=0\), then
\[
(0f)_1(x)=\|0-x\|=\|-x\|,
\qquad
(0f)_2(x)=\|x-0\|=\|x\|.
\]
If \(\alpha<0\), then
\[
(\alpha f)_1(x)=|\alpha| f_2(\alpha^{-1}x),
\qquad
(\alpha f)_2(x)=|\alpha| f_1(\alpha^{-1}x).
\]
\end{definition}

\begin{definition}\label{def:oplusHull}
Let \(f=(f_1,f_2),g=(g_1,g_2)\in\E(X,\|\cdot\|)\).  Define
\[
f\oplus g=\bigl((f\oplus g)_1,(f\oplus g)_2\bigr)
\]
by
\[
(f\oplus g)_1(x)
=
\sup_{s\in X}
\bigl(f_1(x-s)\dot{-}g_2(s)\bigr),
\]
and
\[
(f\oplus g)_2(x)
=
\sup_{s\in X}
\bigl(f_2(x-s)\dot{-}g_1(s)\bigr),
\]
for every \(x\in X\).
\end{definition}

\begin{theorem}\label{thm:Evectorspace}
With the scalar multiplication of Definition~\ref{def:scalarHull} and the
addition of Definition~\ref{def:oplusHull}, the Isbell-convex hull
\(\E(X,\|\cdot\|)\) is a real vector space.  Moreover,
\[
i:X\longrightarrow \E(X,\|\cdot\|),
\qquad
i(a)=f_a,
\]
is a linear embedding.
\end{theorem}

\begin{remark}\label{rem:additiveidentity}
The additive identity of \(\E(X,\|\cdot\|)\) is \(f_0\), where
\[
(f_0)_1(x)=\|0-x\|=\|-x\|,
\qquad
(f_0)_2(x)=\|x-0\|=\|x\|.
\]
The asymmetric norm on the hull is obtained from the Isbell-hull di-metric by
\[
\widetilde p(f):=D(f_0,f).
\]
This asymmetric norm is used in Section~\ref{sec:hullconvex} to prove that the
lifted barycentric map
\[
\mathbb W(f,g,\lambda)=\lambda f\oplus(1-\lambda)g
\]
is a Takahashi convexity structure on the hull.
\end{remark}

\section{The Isbell-hull di-metric and the lifted convexity structure}
\label{sec:hullconvex}

This section contains the main geometric construction of the paper.  We do
not introduce a new quasi-metric on the Isbell-convex hull.  Instead, we use
the canonical Isbell-hull di-metric of Kemajou, K\"unzi and Olela-Otafudu and
combine it with the vector-space operations on the hull constructed by
Conradie, K\"unzi and Olela-Otafudu.  This yields a natural barycentric map
\[
\mathbb W(f,g,\lambda)=\lambda f\oplus(1-\lambda)g
\]
on the hull.  We show that this map is a Takahashi convexity structure on
\((\E(X,\|\cdot\|),q_{\E})\).  Thus the Isbell-convex hull is not only an
injective enlargement of the original asymmetrically normed space, but also
carries an intrinsic convex \(T_{0}\)-quasi-metric geometry.

Throughout this section, let \((X,\|\cdot\|)\) be an asymmetrically normed real
vector space and let
\[
q_{\|\cdot\|}(x,y)=\|x-y\|,\qquad x,y\in X,
\]
be the associated \(T_{0}\)-quasi-metric.  We write
\[
\E(X,\|\cdot\|):=\mathcal E(X,q_{\|\cdot\|})
\]
for the Isbell-convex hull of the corresponding di-space.

\subsection{The canonical Isbell-hull di-metric}

Recall that an element of \(\E(X,\|\cdot\|)\) is a minimal ample pair
\(f=(f_1,f_2)\), where \(f_i:X\to[0,\infty)\).  For \(a\in X\), the canonical
embedded point is
\[
f_a=\bigl((f_a)_1,(f_a)_2\bigr),
\qquad
(f_a)_1(x)=q_{\|\cdot\|}(a,x)=\|a-x\|,
\quad
(f_a)_2(x)=q_{\|\cdot\|}(x,a)=\|x-a\|.
\]
We denote the canonical embedding by
\[
i:X\longrightarrow \E(X,\|\cdot\|),\qquad i(a)=f_a.
\]

\begin{definition}\label{def:qE}
Let \(f=(f_1,f_2),g=(g_1,g_2)\in \E(X,\|\cdot\|)\).  Define
\[
q_{\E}(f,g):=D(f,g),
\]
where \(D\) is the canonical Isbell-hull di-metric
\[
D(f,g)
=
\sup_{x\in X}\bigl(f_1(x)\dot{-}g_1(x)\bigr)
\vee
\sup_{x\in X}\bigl(g_2(x)\dot{-}f_2(x)\bigr),
\qquad
a\dot{-}b:=\max\{a-b,0\}.
\]
For minimal ample pairs, this reduces to
\[
q_{\E}(f,g)
=
\sup_{x\in X}\bigl(f_1(x)\dot{-}g_1(x)\bigr)
=
\sup_{x\in X}\bigl(g_2(x)\dot{-}f_2(x)\bigr).
\]
\end{definition}

\begin{remark}\label{rem:qEorientation}
The orientation in Definition~\ref{def:qE} is important.  The first-component
formula is
\[
q_{\E}(f,g)=\sup_{x\in X}\bigl(f_1(x)\dot{-}g_1(x)\bigr),
\]
not
\[
\sup_{x\in X}\bigl(g_1(x)\dot{-}f_1(x)\bigr).
\]
With this convention, the canonical embedding \(i:X\to\E(X,\|\cdot\|)\) is
isometric for the original quasi-metric \(q_{\|\cdot\|}\).
\end{remark}

\begin{theorem}\label{thm:qEisometric}
The map \(q_{\E}\) is a \(T_{0}\)-quasi-metric on \(\E(X,\|\cdot\|)\).
Moreover, for all \(a,b\in X\),
\[
q_{\E}(f_a,f_b)=q_{\|\cdot\|}(a,b)=\|a-b\|.
\]
Hence the canonical embedding
\[
i:X\to\E(X,\|\cdot\|),\qquad i(a)=f_a,
\]
is isometric.
\end{theorem}

\begin{proof}
The fact that \(q_{\E}\) is a \(T_{0}\)-quasi-metric is part of the
Isbell-hull construction for di-spaces.  The isometry of the canonical
embedding follows from the same construction.  Indeed, for \(a,b\in X\),
\[
q_{\E}(f_a,f_b)=D(f_a,f_b)=q_{\|\cdot\|}(a,b)=\|a-b\|.
\]
\end{proof}

\begin{remark}\label{rem:qEconjugate}
The conjugate quasi-metric is
\[
q_{\E}^{t}(f,g)=q_{\E}(g,f).
\]
Using the minimal-pair formula in Definition~\ref{def:qE}, we have
\[
q_{\E}^{t}(f,g)
=
\sup_{x\in X}\bigl(g_1(x)\dot{-}f_1(x)\bigr)
=
\sup_{x\in X}\bigl(f_2(x)\dot{-}g_2(x)\bigr).
\]
Thus the two components of a minimal ample pair encode the two directed
distances on the hull.
\end{remark}

\subsection{The vector-space structure of the hull}

The Isbell-convex hull \(\E(X,\|\cdot\|)\) carries canonical scalar
multiplication and addition operations.  We denote the addition by \(\oplus\).
With these operations, \(\E(X,\|\cdot\|)\) is a real vector space, and the
canonical embedding \(i:X\to\E(X,\|\cdot\|)\) is linear.

\begin{proposition}\label{prop:compat}
For all \(a,b\in X\) and all \(\alpha\in\mathbb R\),
\[
f_{a+b}=f_a\oplus f_b,
\qquad
f_{\alpha a}=\alpha f_a.
\]
Consequently, \(i:X\to\E(X,\|\cdot\|)\) is affine, and indeed linear, with
respect to the operations \((\oplus,\cdot)\) on the hull.
\end{proposition}

\begin{proof}
This is precisely the compatibility of the canonical embedding with the
vector-space structure on \(\E(X,\|\cdot\|)\).  The operations \(\oplus\) and
scalar multiplication are constructed so that \(x\mapsto f_x\) is a linear
embedding of \(X\) into its Isbell-convex hull.
\end{proof}

Let \(f^0\) denote the additive identity of \(\E(X,\|\cdot\|)\).  Thus
\[
f^0=f_0,
\qquad
(f^0)_1(x)=\|0-x\|=\|-x\|,
\qquad
(f^0)_2(x)=\|x-0\|=\|x\|.
\]
The hull di-metric gives rise to an asymmetric norm on \(\E(X,\|\cdot\|)\) by
\[
\widetilde p(f):=q_{\E}(f^0,f),\qquad f\in\E(X,\|\cdot\|).
\]
Equivalently, the associated quasi-metric of the asymmetrically normed vector
space \((\E(X,\|\cdot\|),\widetilde p)\) is the Isbell-hull di-metric:
\[
q_{\E}(f,g)=\widetilde p(g-f),
\qquad f,g\in\E(X,\|\cdot\|).
\]

\begin{remark}\label{rem:qEfromptilde}
The identity
\[
q_{\E}(f,g)=\widetilde p(g-f)
\]
provides a connection between asymmetric normed linear geometry and Takahashi convexity.:
once the hull is viewed as an asymmetrically normed real vector space, the
usual affine convexity argument applies to \(\E(X,\|\cdot\|)\) itself.
\end{remark}

\subsection{The lifted barycentric map}

\begin{definition}\label{def:Wlift}
For \(f,g\in\E(X,\|\cdot\|)\) and \(\lambda\in[0,1]\), define
\[
\mathbb W(f,g,\lambda):=\lambda f\oplus(1-\lambda)g.
\]
We call \(\mathbb W\) the \emph{lifted Takahashi convexity map} on the
Isbell-convex hull.
\end{definition}

\begin{remark}\label{rem:WliftMotivation}
The formula for \(\mathbb W\) is the hull analogue of the standard affine
convexity map
\[
S(x,y,\lambda)=\lambda x+(1-\lambda)y
\]
on an asymmetrically normed real vector space.  The difference is that the
addition and scalar multiplication are now the canonical operations on the
Isbell-convex hull.
\end{remark}

\subsection{The main hull convexity theorem}

\begin{theorem}\label{thm:HullConvex}
Let \((X,\|\cdot\|)\) be an asymmetrically normed real vector space.  Then
\[
(\E(X,\|\cdot\|),q_{\E},\mathbb W)
\]
is a convex \(T_{0}\)-quasi-metric space in the sense of Takahashi--K\"unzi--Yildiz.
That is, for all \(f,g,h\in\E(X,\|\cdot\|)\) and all \(\lambda\in[0,1]\),
\[
q_{\E}\bigl(h,\mathbb W(f,g,\lambda)\bigr)
\le
\lambda q_{\E}(h,f)+(1-\lambda)q_{\E}(h,g),
\]
and
\[
q_{\E}\bigl(\mathbb W(f,g,\lambda),h\bigr)
\le
\lambda q_{\E}(f,h)+(1-\lambda)q_{\E}(g,h).
\]
\end{theorem}

\begin{proof}
Since \(\E(X,\|\cdot\|)\) is an asymmetrically normed real vector space with
asymmetric norm \(\widetilde p\), and since
\[
q_{\E}(u,v)=\widetilde p(v-u),
\]
we compute, for \(f,g,h\in\E(X,\|\cdot\|)\),
\[
q_{\E}\bigl(h,\mathbb W(f,g,\lambda)\bigr)
=
\widetilde p\bigl(\mathbb W(f,g,\lambda)-h\bigr).
\]
Using the vector-space operations on the hull,
\[
\mathbb W(f,g,\lambda)-h
=
\lambda(f-h)+(1-\lambda)(g-h).
\]
By subadditivity and positive homogeneity of the asymmetric norm
\(\widetilde p\),
\[
\widetilde p\bigl(\lambda(f-h)+(1-\lambda)(g-h)\bigr)
\le
\lambda\widetilde p(f-h)+(1-\lambda)\widetilde p(g-h).
\]
Therefore
\[
q_{\E}\bigl(h,\mathbb W(f,g,\lambda)\bigr)
\le
\lambda q_{\E}(h,f)+(1-\lambda)q_{\E}(h,g).
\]

Similarly,
\[
q_{\E}\bigl(\mathbb W(f,g,\lambda),h\bigr)
=
\widetilde p\bigl(h-\mathbb W(f,g,\lambda)\bigr).
\]
But
\[
h-\mathbb W(f,g,\lambda)
=
\lambda(h-f)+(1-\lambda)(h-g),
\]
and hence
\[
q_{\E}\bigl(\mathbb W(f,g,\lambda),h\bigr)
\le
\lambda q_{\E}(f,h)+(1-\lambda)q_{\E}(g,h).
\]
Thus \(\mathbb W\) satisfies both defining inequalities of a Takahashi
convexity structure on the \(T_{0}\)-quasi-metric space
\((\E(X,\|\cdot\|),q_{\E})\).
\end{proof}

\begin{corollary}\label{cor:WliftSynchronized}
The lifted convexity map \(\mathbb W\) is synchronized:
\[
\mathbb W(f,g,\lambda)=\mathbb W(g,f,1-\lambda)
\]
for all \(f,g\in\E(X,\|\cdot\|)\) and all \(\lambda\in[0,1]\).
\end{corollary}

\begin{proof}
Since \(\oplus\) is the vector-space addition on \(\E(X,\|\cdot\|)\), it is
commutative.  Therefore
\[
\mathbb W(g,f,1-\lambda)
=
(1-\lambda)g\oplus \lambda f
=
\lambda f\oplus(1-\lambda)g
=
\mathbb W(f,g,\lambda).
\]
\end{proof}

\begin{proposition}\label{prop:WliftS}
The lifted convexity structure \(\mathbb W\) has property \({\rm(S)}\).  That
is, for all \(f,g,f',g'\in\E(X,\|\cdot\|)\) and all \(\lambda\in[0,1]\),
\[
q_{\E}\bigl(\mathbb W(f,g,\lambda),\mathbb W(f',g',\lambda)\bigr)
\le
\lambda q_{\E}(f,f')+(1-\lambda)q_{\E}(g,g').
\]
\end{proposition}

\begin{proof}
Using \(q_{\E}(u,v)=\widetilde p(v-u)\), we have
\[
q_{\E}\bigl(\mathbb W(f,g,\lambda),\mathbb W(f',g',\lambda)\bigr)
=
\widetilde p\bigl(\mathbb W(f',g',\lambda)-\mathbb W(f,g,\lambda)\bigr).
\]
By linearity of the hull operations,
\[
\mathbb W(f',g',\lambda)-\mathbb W(f,g,\lambda)
=
\lambda(f'-f)+(1-\lambda)(g'-g).
\]
Hence, by subadditivity and positive homogeneity of \(\widetilde p\),
\[
\widetilde p\bigl(\lambda(f'-f)+(1-\lambda)(g'-g)\bigr)
\le
\lambda\widetilde p(f'-f)+(1-\lambda)\widetilde p(g'-g).
\]
Thus
\[
q_{\E}\bigl(\mathbb W(f,g,\lambda),\mathbb W(f',g',\lambda)\bigr)
\le
\lambda q_{\E}(f,f')+(1-\lambda)q_{\E}(g,g').
\]
\end{proof}

\begin{remark}\label{rem:HullMetricCase}
When the asymmetric norm on \(X\) is symmetric, the symmetrisation
\(q_{\E}^{s}\) is a genuine metric on the hull.  In that case the preceding
construction reduces to the usual Takahashi convexity structure on a convex
metric space, but the present formulation retains the directed information
coming from the asymmetric norm.
\end{remark}

\subsection{Compatibility with the base space}

We now show that the lifted convexity on the hull extends the standard affine
convexity on the original space.

\begin{theorem}\label{thm:Intertwine}
Let \(S\) be the standard affine convexity structure on \(X\), namely
\[
S(x,y,\lambda)=\lambda x+(1-\lambda)y,
\qquad x,y\in X,\ \lambda\in[0,1].
\]
Then for all \(x,y\in X\) and all \(\lambda\in[0,1]\),
\[
i\bigl(S(x,y,\lambda)\bigr)
=
\mathbb W\bigl(i(x),i(y),\lambda\bigr).
\]
Consequently, \(i(X)\) is a \(\mathbb W\)-convex subset of
\(\E(X,\|\cdot\|)\).
\end{theorem}

\begin{proof}
By Proposition~\ref{prop:compat} and Definition~\ref{def:Wlift},
\[
\mathbb W\bigl(i(x),i(y),\lambda\bigr)
=
\lambda f_x\oplus(1-\lambda)f_y.
\]
Since \(i\) is linear with respect to the hull operations,
\[
\lambda f_x\oplus(1-\lambda)f_y
=
f_{\lambda x}\oplus f_{(1-\lambda)y}
=
f_{\lambda x+(1-\lambda)y}.
\]
Therefore
\[
\mathbb W\bigl(i(x),i(y),\lambda\bigr)
=
f_{S(x,y,\lambda)}
=
i\bigl(S(x,y,\lambda)\bigr).
\]
Thus \(i(X)\) is closed under the lifted barycentric operation \(\mathbb W\),
and hence \(i(X)\) is \(\mathbb W\)-convex.
\end{proof}

\begin{remark}\label{rem:structural-upgrade}
The predecessor work shows that minimal pairs in the Isbell-convex hull are
\(W\)-convex as functions on the original space \(X\).  The present theorem
expresses a different and stronger structural statement: the hull itself
carries a Takahashi convexity structure, and the canonical embedding
intertwines the base affine convexity with the lifted hull convexity.
\end{remark}

\section{Functional and intrinsic convexity on the Isbell hull}
\label{sec:pairconvex}

Section~\ref{sec:hullconvex} constructed a Takahashi convexity structure
directly on the Isbell-convex hull \(\E(X,\|\cdot\|)\).  The present section
explains how this intrinsic hull convexity relates to the earlier functional
convexity of minimal pairs on the original space \(X\).

There are two levels of convexity involved.  The first is \emph{functional}:
an element \(f=(f_1,f_2)\in\E(X,\|\cdot\|)\) may be viewed as a pair of
real-valued functions on \(X\), and one may ask whether each component is
convex with respect to a given Takahashi convexity structure \(W\) on \(X\).
This is the point of view developed in \cite{OlelaZweni2023}.  The second is
\emph{intrinsic}: the hull \(\E(X,\|\cdot\|)\) itself carries the lifted
barycentric operation
\[
\mathbb W(f,g,\lambda)=\lambda f\oplus(1-\lambda)g,
\]
and hence becomes a convex \(T_0\)-quasi-metric space in its own right.

The aim of this section is to record the compatibility between these two
levels.  In particular, under the usual translation-invariance and homogeneity
assumptions on \(W\), the class of \(W\)-convex minimal pairs is stable under
the hull operations used to define \(\mathbb W\).

\subsection{\texorpdfstring{\(W\)}{W}-convex function pairs}

\begin{definition}\label{def:Wconvexpair}
Let \((X,\|\cdot\|)\) be an asymmetrically normed real vector space, and let
\[
q_{\|\cdot\|}(x,y)=\|x-y\|
\]
be the associated \(T_0\)-quasi-metric.  Suppose that
\((X,q_{\|\cdot\|},W)\) is a convex \(T_0\)-quasi-metric space.

A pair \(f=(f_1,f_2)\) of extended real-valued functions
\[
f_j:X\longrightarrow (-\infty,\infty],
\qquad j=1,2,
\]
is called \emph{\(W\)-convex} if, for all \(x,y\in X\) and all
\(\lambda\in[0,1]\),
\[
f_j\bigl(W(x,y,\lambda)\bigr)
\le
\lambda f_j(x)+(1-\lambda)f_j(y),
\qquad j=1,2.
\]
\end{definition}

\begin{remark}\label{rem:Wconvexpair-basic}
\begin{enumerate}[label=(\alph*),leftmargin=2.2em]
\item If \(q=q^t\) and \(f_1=f_2=f\), then
Definition~\ref{def:Wconvexpair} reduces to the usual convexity of a
real-valued function with respect to a Takahashi convexity structure.

\item If \(C\subseteq X\) is \(W\)-convex, then the restriction of a
\(W\)-convex pair to \(C\) remains \(W\)-convex with respect to the restricted
convexity structure.
\end{enumerate}
\end{remark}

\begin{example}\label{ex:distpair}
For each \(z\in X\), define the canonical distance pair
\[
f_z=\bigl((f_z)_1,(f_z)_2\bigr)
\]
by
\[
(f_z)_1(x)=\|z-x\|,
\qquad
(f_z)_2(x)=\|x-z\|.
\]
Then \(f_z\) is \(W\)-convex on \(X\).
\end{example}

\begin{proof}
Fix \(x,y,z\in X\) and \(\lambda\in[0,1]\).  Since \(W\) is a Takahashi
convexity structure on \((X,q_{\|\cdot\|})\), we have
\[
q_{\|\cdot\|}\bigl(z,W(x,y,\lambda)\bigr)
\le
\lambda q_{\|\cdot\|}(z,x)+(1-\lambda)q_{\|\cdot\|}(z,y).
\]
Using \(q_{\|\cdot\|}(a,b)=\|a-b\|\), this gives
\[
\|z-W(x,y,\lambda)\|
\le
\lambda\|z-x\|+(1-\lambda)\|z-y\|,
\]
so \((f_z)_1\) is \(W\)-convex.

Similarly, the second Takahashi inequality gives
\[
q_{\|\cdot\|}\bigl(W(x,y,\lambda),z\bigr)
\le
\lambda q_{\|\cdot\|}(x,z)+(1-\lambda)q_{\|\cdot\|}(y,z),
\]
that is,
\[
\|W(x,y,\lambda)-z\|
\le
\lambda\|x-z\|+(1-\lambda)\|y-z\|.
\]
Thus \((f_z)_2\) is \(W\)-convex.
\end{proof}

\subsection{Minimal pairs as \texorpdfstring{\(W\)}{W}-convex functions}

The following result is the functional convexity principle inherited from the
predecessor paper.  It says that, under translation invariance of \(W\), every
minimal pair in the Isbell-convex hull is convex when regarded as a pair of
functions on the original space \(X\).

\begin{theorem}\label{thm:minpairsWconvex}
Let \((X,\|\cdot\|,W)\) be a convex asymmetrically normed real vector space.
Assume that \(W\) is translation-invariant in the sense of
Definition~\ref{def:TIH}.  Then every
\[
f=(f_1,f_2)\in\E(X,\|\cdot\|)
\]
is \(W\)-convex on \(X\).
\end{theorem}

\begin{proof}
This is the minimal-pair convexity result proved in \cite{OlelaZweni2023}.
The proof uses the minimality of \(f\), the sup-representation formulas for
minimal ample pairs, and translation invariance of \(W\).  The essential point
is that a failure of the Jensen-type inequality for either component would
allow one to construct a strictly smaller ample pair, contradicting minimality.
\end{proof}

\begin{remark}\label{rem:functional-vs-intrinsic}
Theorem~\ref{thm:minpairsWconvex} is a functional statement: it concerns the
convexity of the components \(f_1\) and \(f_2\) as functions on \(X\).  By
contrast, Theorem~\ref{thm:HullConvex} is an intrinsic hull statement: it
asserts that the hull itself is a convex \(T_0\)-quasi-metric space under
\(\mathbb W\).  Thus the present paper upgrades the predecessor's functional
convexity principle to a genuine convex geometry on the Isbell-convex hull.
\end{remark}

\subsection{Stability under scalar multiplication}

We next record stability properties of \(W\)-convex minimal pairs under the
algebraic operations of the hull.  These properties explain why the lifted
barycentric operation \(\mathbb W\) is compatible with the functional
convexity inherited from \(X\).

\begin{proposition}\label{prop:scalarWconvex}
Let \((X,\|\cdot\|,W)\) be a convex asymmetrically normed real vector space.
Assume that \(W\) is homogeneous.  If \(f\in\E(X,\|\cdot\|)\) is
\(W\)-convex on \(X\), then \(\alpha f\) is \(W\)-convex on \(X\) for every
\(\alpha\in\mathbb R\).
\end{proposition}

\begin{proof}
Let \(f=(f_1,f_2)\in\E(X,\|\cdot\|)\) be \(W\)-convex.

First suppose \(\alpha>0\).  For \(j=1,2\), the scalar operation on the hull
gives
\[
(\alpha f)_j(x)=\alpha f_j(\alpha^{-1}x).
\]
Using homogeneity of \(W\),
\[
\alpha^{-1}W(x,y,\lambda)
=
W(\alpha^{-1}x,\alpha^{-1}y,\lambda).
\]
Hence
\[
(\alpha f)_j\bigl(W(x,y,\lambda)\bigr)
=
\alpha f_j\bigl(W(\alpha^{-1}x,\alpha^{-1}y,\lambda)\bigr).
\]
Since \(f_j\) is \(W\)-convex,
\[
(\alpha f)_j\bigl(W(x,y,\lambda)\bigr)
\le
\alpha\Bigl[
\lambda f_j(\alpha^{-1}x)
+
(1-\lambda)f_j(\alpha^{-1}y)
\Bigr].
\]
Therefore
\[
(\alpha f)_j\bigl(W(x,y,\lambda)\bigr)
\le
\lambda(\alpha f)_j(x)+(1-\lambda)(\alpha f)_j(y).
\]

If \(\alpha=0\), then \(0f=f_0\), the additive identity of the hull.  Since
\(f_0\) is the canonical distance pair at \(0\), Example~\ref{ex:distpair}
implies that \(f_0\) is \(W\)-convex.

Finally, suppose \(\alpha<0\).  The scalar action swaps the two components:
\[
(\alpha f)_1(x)=|\alpha|f_2(\alpha^{-1}x),
\qquad
(\alpha f)_2(x)=|\alpha|f_1(\alpha^{-1}x).
\]
The same homogeneity argument used above, together with the \(W\)-convexity of
both \(f_1\) and \(f_2\), gives the desired inequalities for both components
of \(\alpha f\).
\end{proof}

\subsection{Stability under the hull addition}

\begin{proposition}\label{prop:oplusWconvex}
Let \((X,\|\cdot\|,W)\) be a convex asymmetrically normed real vector space.
Assume that \(W\) is translation-invariant.  If
\(f,g\in\E(X,\|\cdot\|)\) are \(W\)-convex on \(X\), then
\(f\oplus g\) is \(W\)-convex on \(X\).
\end{proposition}

\begin{proof}
Let \(f=(f_1,f_2)\) and \(g=(g_1,g_2)\).  For \(j=1,2\), set
\[
j' = 3-j.
\]
Thus \(1'=2\) and \(2'=1\).  The hull addition may be written uniformly as
\[
(f\oplus g)_j(x)
=
\sup_{z\in X}
\bigl(f_j(x-z)\dot{-}g_{j'}(z)\bigr).
\]
Indeed, for \(j=1\) this says
\[
(f\oplus g)_1(x)
=
\sup_{z\in X}
\bigl(f_1(x-z)\dot{-}g_2(z)\bigr),
\]
while for \(j=2\) it says
\[
(f\oplus g)_2(x)
=
\sup_{z\in X}
\bigl(f_2(x-z)\dot{-}g_1(z)\bigr).
\]

Fix \(x,y\in X\) and \(\lambda\in[0,1]\).  Translation invariance gives
\[
W(x,y,\lambda)-z
=
W(x-z,y-z,\lambda)
\qquad (z\in X).
\]
Therefore
\[
(f\oplus g)_j\bigl(W(x,y,\lambda)\bigr)
=
\sup_{z\in X}
\bigl(
f_j(W(x,y,\lambda)-z)\dot{-}g_{j'}(z)
\bigr)
\]
\[
=
\sup_{z\in X}
\bigl(
f_j(W(x-z,y-z,\lambda))\dot{-}g_{j'}(z)
\bigr).
\]
Since \(f_j\) is \(W\)-convex,
\[
f_j(W(x-z,y-z,\lambda))
\le
\lambda f_j(x-z)+(1-\lambda)f_j(y-z).
\]
Hence
\[
(f\oplus g)_j\bigl(W(x,y,\lambda)\bigr)
\le
\sup_{z\in X}
\Bigl(
\lambda f_j(x-z)+(1-\lambda)f_j(y-z)\dot{-}g_{j'}(z)
\Bigr).
\]

Using the elementary inequality
\[
\bigl(\lambda A+(1-\lambda)B-C\bigr)^+
\le
\lambda(A-C)^+ + (1-\lambda)(B-C)^+
\]
for \(A,B,C\in\mathbb R\), we obtain
\[
(f\oplus g)_j\bigl(W(x,y,\lambda)\bigr)
\le
\lambda
\sup_{z\in X}\bigl(f_j(x-z)\dot{-}g_{j'}(z)\bigr)
+
(1-\lambda)
\sup_{z\in X}\bigl(f_j(y-z)\dot{-}g_{j'}(z)\bigr).
\]
Therefore
\[
(f\oplus g)_j\bigl(W(x,y,\lambda)\bigr)
\le
\lambda (f\oplus g)_j(x)+(1-\lambda)(f\oplus g)_j(y).
\]
Since this holds for \(j=1,2\), the pair \(f\oplus g\) is \(W\)-convex.
\end{proof}

\subsection{Affine stability and compatibility with the lifted structure}

Combining the preceding two propositions gives the stability of
\(W\)-convexity under the affine operations used to define the lifted hull
convexity.

\begin{corollary}\label{cor:affineWconvex}
Let \((X,\|\cdot\|,W)\) be a convex asymmetrically normed real vector space.
Assume that \(W\) is translation-invariant and homogeneous.  If
\(f,g\in\E(X,\|\cdot\|)\) are \(W\)-convex on \(X\), then
\[
\mathbb W(f,g,\lambda)
=
\lambda f\oplus(1-\lambda)g
\]
is \(W\)-convex on \(X\) for every \(\lambda\in[0,1]\).
\end{corollary}

\begin{proof}
By Proposition~\ref{prop:scalarWconvex}, the pairs \(\lambda f\) and
\((1-\lambda)g\) are \(W\)-convex.  By
Proposition~\ref{prop:oplusWconvex}, their hull sum
\[
\lambda f\oplus(1-\lambda)g
\]
is also \(W\)-convex.
\end{proof}

\begin{corollary}\label{cor:allminimalstable}
Assume that \(W\) is translation-invariant and homogeneous.  Then every
element of \(\E(X,\|\cdot\|)\) is \(W\)-convex on \(X\), and the class
\(\E(X,\|\cdot\|)\) is stable under the lifted barycentric operation
\[
(f,g,\lambda)\longmapsto \mathbb W(f,g,\lambda).
\]
\end{corollary}

\begin{proof}
By Theorem~\ref{thm:minpairsWconvex}, every element of
\(\E(X,\|\cdot\|)\) is \(W\)-convex on \(X\).  The stability under
\(\mathbb W\) follows from Corollary~\ref{cor:affineWconvex}.
\end{proof}

\begin{remark}\label{rem:two-level-compatibility}
Corollary~\ref{cor:allminimalstable} makes precise the compatibility between
the two levels of convexity considered in this paper.  On the one hand,
Theorem~\ref{thm:HullConvex} shows that \(\mathbb W\) is an intrinsic
Takahashi convexity structure on the hull.  On the other hand, the present
corollary shows that the same barycentric operation preserves the functional
\(W\)-convexity of minimal pairs on \(X\).  Thus the lifted hull geometry is
consistent with the earlier function-pair convexity theory.
\end{remark}

\section{Segments in the hull and uniqueness of the lifted structure}
\label{sec:segments}

In Section~\ref{sec:hullconvex} we showed that the Isbell-convex hull
\[
(\E(X,\|\cdot\|),q_{\E})
\]
carries the lifted Takahashi convexity structure
\[
\mathbb W(f,g,\lambda)=\lambda f\oplus(1-\lambda)g.
\]
We now record the corresponding segment theory.  The purpose of this section is
not to develop a new segment theorem from scratch, but to apply the general
theory of convexity structures in \(T_0\)-quasi-metric spaces to the specific
hull geometry constructed above.

The main point is that, whenever the lifted structure \(\mathbb W\) is unique,
the segment joining two hull elements is isometrically parametrised by a
standard directed interval.  Thus uniqueness turns the formal barycentric path
\[
\lambda\longmapsto \mathbb W(f,g,\lambda)
\]
into a genuine quasi-metric segment.

\subsection{Hull segments}

\begin{definition}\label{def:hullsegments}
Let \(f,g\in\E(X,\|\cdot\|)\).  The \emph{\(\mathbb W\)-segment} from \(f\)
to \(g\) is
\[
S[f,g]
:=
\{\mathbb W(f,g,\lambda):\lambda\in[0,1]\}.
\]
Equivalently,
\[
S[f,g]
=
\{\lambda f\oplus(1-\lambda)g:\lambda\in[0,1]\}.
\]
\end{definition}

\begin{remark}\label{rem:hullsegments-basic}
The segment \(S[f,g]\) is the image of the unit interval under the map
\[
[0,1]\longrightarrow \E(X,\|\cdot\|),
\qquad
\lambda\longmapsto \lambda f\oplus(1-\lambda)g.
\]
Since \((\E(X,\|\cdot\|),q_{\E},\mathbb W)\) is a convex
\(T_0\)-quasi-metric space by Theorem~\ref{thm:HullConvex}, every
\(\mathbb W\)-convex subset of the hull that contains \(f\) and \(g\) also
contains \(S[f,g]\).  If \(f=g\), then \(S[f,g]=\{f\}\).
\end{remark}

\subsection{Uniqueness of the lifted convexity structure}

\begin{definition}\label{def:uniqW}
We say that the lifted convexity structure \(\mathbb W\) is \emph{unique} on
\((\E(X,\|\cdot\|),q_{\E})\) if the following condition holds.

Whenever \(h\in\E(X,\|\cdot\|)\), \(f,g\in\E(X,\|\cdot\|)\), and
\(\lambda\in[0,1]\) satisfy
\[
q_{\E}(u,h)
\le
\lambda q_{\E}(u,f)+(1-\lambda)q_{\E}(u,g),
\]
and
\[
q_{\E}(h,u)
\le
\lambda q_{\E}(f,u)+(1-\lambda)q_{\E}(g,u)
\]
for every \(u\in\E(X,\|\cdot\|)\), then
\[
h=\mathbb W(f,g,\lambda).
\]
\end{definition}

\begin{remark}\label{rem:uniqW-context}
This is the uniqueness notion for Takahashi convexity structures in
\(T_0\)-quasi-metric spaces, specialised to the Isbell-convex hull.  The
condition is not automatic.  It is a genuine geometric restriction on the
directed hull geometry.  In this section, uniqueness is used only as a
hypothesis for the isometric parametrisation of hull segments.
\end{remark}

\subsection{Directed intervals}

For \(\alpha,\beta\in[0,\infty)\), with \(\alpha+\beta\neq0\), let
\(I_{\alpha,\beta}\) denote the unit interval \([0,1]\) equipped with the
\(T_0\)-quasi-metric
\[
d_{\alpha,\beta}(s,t)
=
\begin{cases}
(s-t)\alpha, & s\ge t,\\[2mm]
(t-s)\beta, & s<t.
\end{cases}
\]
This is the standard directed interval associated with the pair
\((\alpha,\beta)\).  The forward length is controlled by \(\alpha\), while the
reverse length is controlled by \(\beta\).

\subsection{Isometric parametrisation of hull segments}

\begin{theorem}\label{thm:segmentIsom}
Assume that the lifted convexity structure \(\mathbb W\) is unique on
\((\E(X,\|\cdot\|),q_{\E})\).  Let \(f,g\in\E(X,\|\cdot\|)\), \(f\neq g\), and
set
\[
\alpha:=q_{\E}(f,g),
\qquad
\beta:=q_{\E}(g,f).
\]
Then the map
\[
\gamma_{f,g}:I_{\alpha,\beta}\longrightarrow (\E(X,\|\cdot\|),q_{\E}),
\qquad
\gamma_{f,g}(\lambda)=\mathbb W(f,g,\lambda),
\]
is an isometric embedding.  Equivalently, for all
\(\lambda,\lambda'\in[0,1]\),
\[
q_{\E}\bigl(\mathbb W(f,g,\lambda),\mathbb W(f,g,\lambda')\bigr)
=
\begin{cases}
(\lambda-\lambda')q_{\E}(f,g), & \lambda\ge \lambda',\\[2mm]
(\lambda'-\lambda)q_{\E}(g,f), & \lambda<\lambda'.
\end{cases}
\]
\end{theorem}

\begin{proof}
By Theorem~\ref{thm:HullConvex}, the triple
\[
(\E(X,\|\cdot\|),q_{\E},\mathbb W)
\]
is a convex \(T_0\)-quasi-metric space.  Under the additional assumption that
\(\mathbb W\) is unique, the general segment theorem for unique convexity
structures in \(T_0\)-quasi-metric spaces applies.  It states that, for any
two distinct points \(a,b\) in a \(T_0\)-quasi-metric space with unique TCS,
the map
\[
\lambda\longmapsto W(a,b,\lambda)
\]
is an isometric embedding of the directed interval
\[
I_{d(a,b),d(b,a)}
\]
into the space.

Applying this result with
\[
a=f,\qquad b=g,\qquad d=q_{\E},\qquad W=\mathbb W,
\]
we obtain that
\[
\lambda\longmapsto \mathbb W(f,g,\lambda)
\]
is an isometric embedding of
\[
I_{q_{\E}(f,g),q_{\E}(g,f)}
\]
into \((\E(X,\|\cdot\|),q_{\E})\).  This is exactly the stated formula.
\end{proof}

\begin{remark}\label{rem:segmentIsom-consequences}
Under uniqueness, every nontrivial hull segment \(S[f,g]\) is quasi-isometric
to a standard directed interval.  The asymmetry of the segment is measured by
the two numbers
\[
q_{\E}(f,g)
\qquad\text{and}\qquad
q_{\E}(g,f).
\]
If these two numbers are unequal, then the segment has different forward and
backward lengths.  In particular, if \(q_{\E}(g,f)=0\) and \(q_{\E}(f,g)>0\),
then the segment behaves like an oriented interval modelled on the standard
quasi-metric \(u(s,t)=s\dot{-}t\).
\end{remark}

\subsection{A uniqueness question for lifted hull convexity}

The preceding theorem shows that uniqueness of \(\mathbb W\) has strong
geometric consequences.  However, uniqueness of the lifted convexity structure
is not automatic.  It should therefore be regarded as an additional rigidity
property of the Isbell-convex hull.

\paragraph{Open problem.}
Characterise those asymmetrically normed real vector spaces
\((X,\|\cdot\|)\) for which the lifted Takahashi convexity structure
\[
\mathbb W(f,g,\lambda)=\lambda f\oplus(1-\lambda)g
\]
is the unique Takahashi convexity structure on
\[
(\E(X,\|\cdot\|),q_{\E}).
\]

\begin{remark}\label{rem:unique-programme}
The preceding open problem is a natural continuation of the uniqueness theory
for convexity structures in \(T_0\)-quasi-metric spaces.  A positive answer in
important classes of asymmetrically normed spaces would show that the
Isbell-convex hull does not merely admit a lifted convexity structure, but
that this structure is forced by the directed metric geometry of the hull.
\end{remark}

\section{Fixed point consequences on the hull}\label{sec:fixhull}

In this section we establish fixed point consequences on the hull
\[
(\E(X,\|\cdot\|),\,q_{\E},\,\mathbb{W})
\]
by adapting the Chebyshev-centre/normal-structure method from convex $T_{0}$-quasi-metric spaces.
The proofs follow the quasi-metric setting of K\"unzi and Yildiz \cite{KunziYildiz2016} and the
fixed point arguments in \cite{Olela2015,OlelaZweni2023}, extending Takahashi's classical theory
for convex metric spaces \cite{Takahashi1970}.

\subsection{Double closure, property (H), and Chebyshev centres}

\begin{definition}\label{def:doubleclosureHull}
For $A\subseteq \E(X,\|\cdot\|)$ define the \emph{double closure} by
\[
\cl_{\tau(q_{\E})}A\ \cap\ \cl_{\tau(q_{\E}^{t})}A,
\]
where $\tau(q_{\E})$ and $\tau(q_{\E}^{t})$ denote the topologies induced by $q_{\E}$ and its
conjugate $q_{\E}^{t}(f,g)=q_{\E}(g,f)$. We call $A$ \emph{doubly closed} if it equals its
double closure.
\end{definition}

\begin{definition}\label{def:Hhull}
We say that $(\E,q_{\E},\mathbb{W})$ has \emph{property {\rm(H)}} if every totally ordered family
$\{C_{i}\}_{i\in I}$ of nonempty, bounded, doubly closed and $\mathbb{W}$-convex subsets of $\E$
(with $C_{j}\subseteq C_{i}$ whenever $i\le j$) has nonempty intersection:
\[
\bigcap_{i\in I} C_{i}\neq\varnothing.
\]
\end{definition}

\begin{definition}\label{def:ChebHull}
Let $A\subseteq \E$ be nonempty and bounded. For $f\in \E$ define
\[
r_f(A)_{q_{\E}}:=\sup_{g\in A}q_{\E}(f,g),\qquad
r_f(A)_{q_{\E}^{t}}:=\sup_{g\in A}q_{\E}(g,f),
\]
\[
r_f(A):=r_f(A)_{q_{\E}}\vee r_f(A)_{q_{\E}^{t}},\qquad
r(A):=\inf_{f\in A} r_f(A),
\]
and
\[
\diam(A):=\sup\{\,q_{\E}(f,g): f,g\in A\,\}.
\]
The \emph{Chebyshev centre} of $A$ is the set
\[
C(A):=\{f\in A:\ r_f(A)=r(A)\}.
\]
\end{definition}

\begin{remark}\label{rem:ChebHull-remarks}
\begin{enumerate}[label=(\alph*),leftmargin=2.2em]
\item By boundedness of $A$, all quantities in Definition~\ref{def:ChebHull} are finite.
\item Throughout, $\diam(A)$ is taken with respect to $q_{\E}$ (not $q_{\E}^{s}$), in line with
the directed quasi-metric setting; compare \cite{KunziYildiz2016,OlelaZweni2023}.
\end{enumerate}
\end{remark}

\subsection{Chebyshev centres: nonemptiness and $\mathbb{W}$-convexity}

\begin{theorem}\label{thm:ChebConvex}
Assume $(\E,q_{\E},\mathbb{W})$ has property {\rm(H)}.
If $A\subseteq \E$ is nonempty, bounded, doubly closed and $\mathbb{W}$-convex, then the Chebyshev centre
$C(A)$ is nonempty, bounded, doubly closed and $\mathbb{W}$-convex.
\end{theorem}

\begin{proof}
Fix $n\in\mathbb{N}$ and set
\[
C_n:=\Bigl\{\,f\in A:\ r_f(A)\le r(A)+\tfrac{1}{n}\,\Bigr\}.
\]
Each $C_n$ is nonempty since $r(A)$ is an infimum, and $C_{n+1}\subseteq C_n$.

To verify that $C_n$ is bounded, fix $f_0\in C_n$. For any $f\in C_n$ and any $g\in A$, the triangle inequality gives
\[
q_{\E}(f_0,f)\le q_{\E}(f_0,g)+q_{\E}(g,f)\le r_{f_0}(A)_{q_{\E}}+r_f(A)_{q_{\E}^{t}}
\le r_{f_0}(A)+r_f(A)\le 2\Bigl(r(A)+\tfrac1n\Bigr).
\]

Next, $C_n$ is $\mathbb{W}$-convex. Let $f,g\in C_n$ and $\lambda\in[0,1]$. For any $h\in A$,
Theorem~\ref{thm:HullConvex} yields
\[
q_{\E}\bigl(\mathbb{W}(f,g,\lambda),h\bigr)\le \lambda q_{\E}(f,h)+(1-\lambda)q_{\E}(g,h),
\]
and
\[
q_{\E}\bigl(h,\mathbb{W}(f,g,\lambda)\bigr)\le \lambda q_{\E}(h,f)+(1-\lambda)q_{\E}(h,g).
\]
Taking suprema over $h\in A$ gives
\[
r_{\mathbb{W}(f,g,\lambda)}(A)_{q_{\E}}
\le \lambda r_f(A)_{q_{\E}}+(1-\lambda)r_g(A)_{q_{\E}},
\qquad
r_{\mathbb{W}(f,g,\lambda)}(A)_{q_{\E}^{t}}
\le \lambda r_f(A)_{q_{\E}^{t}}+(1-\lambda)r_g(A)_{q_{\E}^{t}},
\]
hence
\[
r_{\mathbb{W}(f,g,\lambda)}(A)\le \lambda r_f(A)+(1-\lambda)r_g(A)\le r(A)+\tfrac1n.
\]
Since $A$ is $\mathbb{W}$-convex and $f,g\in A$, we have $\mathbb{W}(f,g,\lambda)\in A$, and therefore
$\mathbb{W}(f,g,\lambda)\in C_n$.

We now show that $C_n$ is doubly closed. Let $(f_\gamma)$ be a net in $C_n$ converging to $f$ in both
$\tau(q_{\E})$ and $\tau(q_{\E}^{t})$, so $q_{\E}(f,f_\gamma)\to 0$ and $q_{\E}(f_\gamma,f)\to 0$.
Since $A$ is doubly closed and each $f_\gamma\in A$, we have $f\in A$. For any $h\in A$,
\[
q_{\E}(f,h)\le q_{\E}(f,f_\gamma)+q_{\E}(f_\gamma,h),\qquad
q_{\E}(h,f)\le q_{\E}(h,f_\gamma)+q_{\E}(f_\gamma,f),
\]
and taking suprema over $h\in A$ yields
\[
r_f(A)\le \bigl(q_{\E}(f,f_\gamma)\vee q_{\E}(f_\gamma,f)\bigr)+r_{f_\gamma}(A)
\le \bigl(q_{\E}(f,f_\gamma)\vee q_{\E}(f_\gamma,f)\bigr)+r(A)+\tfrac1n.
\]
Letting $\gamma$ tend along the net gives $r_f(A)\le r(A)+\tfrac1n$, hence $f\in C_n$.

Finally, property {\rm(H)} implies $\bigcap_{n\ge 1}C_n\neq\varnothing$. Any $f$ in this intersection satisfies
$r_f(A)\le r(A)+\frac1n$ for all $n$, hence $r_f(A)=r(A)$ and $f\in C(A)$. Therefore $C(A)\neq\varnothing$.
Moreover, $C(A)=\bigcap_{n\ge 1}C_n$ is bounded, doubly closed and $\mathbb{W}$-convex as an intersection of
sets with these properties.
\end{proof}

\subsection{Normal structure and fixed points}

\begin{definition}\label{def:NormalHull}
We say that $(\E,q_{\E},\mathbb{W})$ has \emph{normal structure} if for every nonempty, bounded,
doubly closed, $\mathbb{W}$-convex set $A\subseteq \E$ with $\diam(A)>0$ one has
\[
r(A)<\diam(A).
\]
\end{definition}

\begin{theorem}\label{thm:FixHull}
Assume $(\E,q_{\E},\mathbb{W})$ has property {\rm(H)} and normal structure.
Let $K\subseteq \E$ be nonempty, bounded, doubly closed and $\mathbb{W}$-convex.
If $T:(K,q_{\E})\to (K,q_{\E})$ is nonexpansive, then $T$ has a fixed point in $K$.
\end{theorem}

\begin{proof}
Consider the family
\[
\Gamma:=\Bigl\{\,D\subseteq K:\ D\neq\varnothing,\ D\ \text{bounded, doubly closed, $\mathbb{W}$-convex, and } T(D)\subseteq D\,\Bigr\}
\]
ordered by inclusion. Clearly $K\in\Gamma$. If $\{D_i\}_{i\in I}$ is a chain in $\Gamma$, then
$D:=\bigcap_{i\in I}D_i$ is nonempty by property {\rm(H)}, and it is bounded, doubly closed,
$\mathbb{W}$-convex and $T$-invariant; hence $D\in\Gamma$. By Zorn's lemma, $\Gamma$ has a minimal element $A$.

We claim that $\diam(A)=0$. Suppose $\diam(A)>0$. Since $A\in\Gamma$, the set $A$ is nonempty, bounded,
doubly closed and $\mathbb{W}$-convex, so by Theorem~\ref{thm:ChebConvex} the Chebyshev centre $C(A)$ is
nonempty, bounded, doubly closed and $\mathbb{W}$-convex. We show $T(C(A))\subseteq C(A)$.
Let $f\in C(A)$ and $g\in A$. Nonexpansiveness gives
\[
q_{\E}\bigl(T(f),T(g)\bigr)\le q_{\E}(f,g),\qquad q_{\E}\bigl(T(g),T(f)\bigr)\le q_{\E}(g,f).
\]
Taking suprema over $g\in A$ and using $T(A)\subseteq A$ yields $r_{T(f)}(A)\le r_f(A)=r(A)$.
Since $T(f)\in A$, also $r(A)\le r_{T(f)}(A)$, so $r_{T(f)}(A)=r(A)$ and hence $T(f)\in C(A)$.

Normal structure gives $r(A)<\diam(A)$, so there exist $a,b\in A$ with $q_{\E}(a,b)>r(A)$.
Then $r_a(A)\ge q_{\E}(a,b)>r(A)$, hence $a\notin C(A)$ and therefore $C(A)\subsetneq A$.
Since $C(A)$ is nonempty, bounded, doubly closed, $\mathbb{W}$-convex and $T$-invariant, we have $C(A)\in\Gamma$,
contradicting the minimality of $A$. Hence $\diam(A)=0$.

Pick any $a\in A$. Then $q_{\E}(a,b)=0=q_{\E}(b,a)$ for all $b\in A$, so $A=\{a\}$ by the $T_{0}$ property
of $q_{\E}$ (Theorem~\ref{thm:qEisometric}). Since $T(A)\subseteq A$, we have $T(a)=a$.
\end{proof}

\begin{theorem}\label{thm:CommonFix}
Assume $(\E,q_{\E},\mathbb{W})$ has property {\rm(H)} and normal structure, and that $\mathbb{W}$
is a unique convex structure on $(\E,q_{\E})$ (Definition~\ref{def:uniqW}).
Let $K\subseteq \E$ be nonempty, bounded, doubly closed and $\mathbb{W}$-convex.
Then every commuting finite family $\{T_1,\dots,T_n\}$ of nonexpansive self-maps on $(K,q_{\E})$
has a nonempty common fixed point set:
\[
\bigcap_{i=1}^{n}\Fix(T_i)\neq\varnothing.
\]
\end{theorem}

\begin{proof}
Fix a nonexpansive map $T$ on $K$. By Theorem~\ref{thm:FixHull}, $\Fix(T)\neq\varnothing$.
Clearly $\Fix(T)$ is bounded as a subset of $K$. To see that $\Fix(T)$ is doubly closed, let $(x_\gamma)$
be a net in $\Fix(T)$ converging to $x$ in both $\tau(q_{\E})$ and $\tau(q_{\E}^{t})$, so
$q_{\E}(x,x_\gamma)\to 0$ and $q_{\E}(x_\gamma,x)\to 0$. Since $x_\gamma=T(x_\gamma)$,
\[
q_{\E}(x,T(x))\le q_{\E}(x,x_\gamma)+q_{\E}(x_\gamma,T(x_\gamma))+q_{\E}(T(x_\gamma),T(x))
\le q_{\E}(x,x_\gamma)+q_{\E}(x_\gamma,x),
\]
where the last inequality uses nonexpansiveness. Letting $\gamma$ tend gives $q_{\E}(x,T(x))=0$.
Similarly,
\[
q_{\E}(T(x),x)\le q_{\E}(T(x),T(x_\gamma))+q_{\E}(T(x_\gamma),x)
\le q_{\E}(x,x_\gamma)+q_{\E}(x_\gamma,x)\to 0,
\]
so $q_{\E}(T(x),x)=0$. By the $T_{0}$ property of $q_{\E}$, we obtain $T(x)=x$, hence $x\in\Fix(T)$.

We show that $\Fix(T)$ is $\mathbb{W}$-convex. Let $f,g\in\Fix(T)$ and $\lambda\in[0,1]$.
For any $u\in K$, nonexpansiveness and Theorem~\ref{thm:HullConvex} yield
\[
q_{\E}\bigl(u,\,T(\mathbb{W}(f,g,\lambda))\bigr)
\le q_{\E}\bigl(u,\mathbb{W}(f,g,\lambda)\bigr)
\le \lambda q_{\E}(u,f)+(1-\lambda)q_{\E}(u,g),
\]
and similarly
\[
q_{\E}\bigl(T(\mathbb{W}(f,g,\lambda)),u\bigr)
\le q_{\E}\bigl(\mathbb{W}(f,g,\lambda),u\bigr)
\le \lambda q_{\E}(f,u)+(1-\lambda)q_{\E}(g,u).
\]
Thus $T(\mathbb{W}(f,g,\lambda))$ satisfies the two inequalities in Definition~\ref{def:uniqW} with endpoints
$f,g$ and parameter $\lambda$. By uniqueness of $\mathbb{W}$,
\[
T(\mathbb{W}(f,g,\lambda))=\mathbb{W}(f,g,\lambda),
\]
so $\mathbb{W}(f,g,\lambda)\in\Fix(T)$.

Now let $T_1,\dots,T_n$ commute. Put $F_1=\Fix(T_1)$, which is nonempty, bounded, doubly closed and
$\mathbb{W}$-convex. Since $T_2$ commutes with $T_1$, we have $T_2(F_1)\subseteq F_1$, hence $T_2$ restricts
to a nonexpansive self-map of $F_1$. Applying Theorem~\ref{thm:FixHull} to $F_1$ yields
$F_1\cap\Fix(T_2)\neq\varnothing$. Proceeding inductively gives $\bigcap_{i=1}^n \Fix(T_i)\neq\varnothing$.
\end{proof}

\begin{remark}\label{rem:FixHull-metric}
If one works with the symmetrisation $q_{\E}^{s}$, then $(\E,q_{\E}^{s},\mathbb{W})$ is a convex
metric space (Proposition~\ref{prop:qts}), and Theorem~\ref{thm:FixHull} reduces to Takahashi-type
fixed point results in convex metric spaces \cite{Takahashi1970}. The contribution here is that the
same fixed point argument can be carried out at the directed level, provided one uses double closure
and normal structure in the quasi-metric sense.
\end{remark}

\section{Examples and counterexamples}\label{sec:examples}

We conclude with two brief illustrations.  The first shows how the lifted
convexity behaves in the standard directed real line.  The second shows that
convexity for the symmetrised metric does not automatically imply convexity
for the underlying \(T_{0}\)-quasi-metric, even after passing to the
Isbell-convex hull.

\subsection{The standard directed real line}

\begin{example}\label{ex:Rstandard}
Let \(X=\mathbb R\) be equipped with the standard \(T_{0}\)-quasi-metric
\[
u(a,b):=a\dot{-}b=\max\{a-b,0\},
\]
and let
\[
S(x,y,\lambda)=\lambda x+(1-\lambda)y
\]
be the usual affine convexity map.  Then \((\mathbb R,u,S)\) is a convex
\(T_{0}\)-quasi-metric space.

Indeed, for \(z,x,y\in\mathbb R\) and \(\lambda\in[0,1]\),
\[
u\bigl(z,S(x,y,\lambda)\bigr)
=
\bigl(z-(\lambda x+(1-\lambda)y)\bigr)^{+}
\]
\[
=
\bigl(\lambda(z-x)+(1-\lambda)(z-y)\bigr)^{+}
\le
\lambda(z-x)^{+}+(1-\lambda)(z-y)^{+}.
\]
Thus
\[
u\bigl(z,S(x,y,\lambda)\bigr)
\le
\lambda u(z,x)+(1-\lambda)u(z,y).
\]
Similarly,
\[
u\bigl(S(x,y,\lambda),z\bigr)
=
\bigl((\lambda x+(1-\lambda)y)-z\bigr)^{+}
\le
\lambda(x-z)^{+}+(1-\lambda)(y-z)^{+},
\]
and hence
\[
u\bigl(S(x,y,\lambda),z\bigr)
\le
\lambda u(x,z)+(1-\lambda)u(y,z).
\]
Therefore \(S\) is a Takahashi convexity structure on \((\mathbb R,u)\).

Equivalently, \(u\) is induced by the asymmetric norm
\[
\|t\|=t^{+}:=\max\{t,0\},
\qquad t\in\mathbb R,
\]
since
\[
u(a,b)=\|a-b\|.
\]
The lifted hull convexity on \(\E(\mathbb R,\|\cdot\|)\) is therefore
\[
\mathbb W(f,g,\lambda)=\lambda f\oplus(1-\lambda)g.
\]
Moreover, the canonical embedding
\[
i:\mathbb R\longrightarrow \E(\mathbb R,\|\cdot\|)
\]
intertwines the affine convexity on \(\mathbb R\) with the lifted convexity on
the hull:
\[
i\bigl(S(x,y,\lambda)\bigr)
=
\mathbb W\bigl(i(x),i(y),\lambda\bigr),
\qquad x,y\in\mathbb R,\ \lambda\in[0,1].
\]
This is precisely the compatibility asserted in
Theorem~\ref{thm:Intertwine}.
\end{example}

\subsection{A counterexample after symmetrisation}

\begin{proposition}\label{prop:counter}
There exist an asymmetrically normed real vector space \((Y,\|\cdot\|)\), a
subset
\[
K\subseteq \E(Y,\|\cdot\|)
\]
and a map
\[
V:K\times K\times[0,1]\longrightarrow K
\]
such that \(V\) is a Takahashi convexity structure on the metric space
\[
(K,q_{\E}^{s}),
\]
but \(V\) is not a Takahashi convexity structure on the \(T_{0}\)-quasi-metric
space
\[
(K,q_{\E}).
\]
\end{proposition}

\begin{proof}
We use the example of K\"unzi and Yildiz showing that a map may be a
Takahashi convexity structure for the symmetrised metric but fail to be one
for the underlying \(T_{0}\)-quasi-metric.

Let \(Y=\mathbb R^{2}\) be equipped with the asymmetric norm
\[
\|(a,b)\|=\max\{a^{+},b^{+}\},
\qquad (a,b)\in\mathbb R^{2}.
\]
Let
\[
q(y_1,y_2)=\|y_1-y_2\|
\]
be the induced \(T_{0}\)-quasi-metric, and let
\[
C=[0,1]\times[0,1]\subseteq Y.
\]
K\"unzi and Yildiz exhibit a map
\[
W:C\times C\times[0,1]\longrightarrow C
\]
which is a Takahashi convexity structure on the symmetrised metric space
\[
(C,q^{s}),
\]
but is not a Takahashi convexity structure on the quasi-metric space
\[
(C,q).
\]

Now embed \(Y\) into its Isbell-convex hull by the canonical isometric
embedding
\[
i:Y\longrightarrow \E(Y,\|\cdot\|).
\]
Put
\[
K:=i(C)\subseteq \E(Y,\|\cdot\|).
\]
Define
\[
V:K\times K\times[0,1]\longrightarrow K
\]
by transporting \(W\) along \(i\):
\[
V(i(x),i(y),\lambda):=i\bigl(W(x,y,\lambda)\bigr),
\qquad x,y\in C,\ \lambda\in[0,1].
\]
Since \(i\) is an isometry from \((Y,q)\) into
\((\E(Y,\|\cdot\|),q_{\E})\), its restriction identifies
\[
(C,q)
\quad\text{with}\quad
(K,q_{\E}|_{K\times K}),
\]
and also identifies
\[
(C,q^{s})
\quad\text{with}\quad
(K,q_{\E}^{s}|_{K\times K}).
\]
Therefore \(V\) is a Takahashi convexity structure on the metric space
\((K,q_{\E}^{s})\).

If \(V\) were also a Takahashi convexity structure on the quasi-metric space
\((K,q_{\E})\), then transporting it back through the isometry \(i\) would make
\(W\) a Takahashi convexity structure on \((C,q)\).  This contradicts the
K\"unzi--Yildiz example.  Hence \(V\) fails to be a Takahashi convexity
structure on \((K,q_{\E})\).
\end{proof}

\begin{remark}\label{rem:counter-meaning}
Proposition~\ref{prop:counter} shows that the two directed inequalities in
the definition of a Takahashi convexity structure cannot be replaced by
ordinary convexity for the symmetrised metric.  The directed geometry of
\(q_{\E}\) contains information that is lost after passing to \(q_{\E}^{s}\).
\end{remark}

\section{Concluding remarks and further directions}
\label{sec:conclusion}

We have shown that the Isbell-convex hull of an asymmetrically normed real
vector space carries a natural lifted Takahashi convexity structure.  The
construction uses two canonical ingredients already present in the hull: the
Isbell-hull di-metric \(q_{\E}\) and the vector-space operations
\((\oplus,\cdot)\).  Thus the barycentric formula
\[
\mathbb W(f,g,\lambda)=\lambda f\oplus(1-\lambda)g
\]
does not impose an external convexity structure on the hull; rather, it is
intrinsic to the algebraic and directed metric structure of
\(\E(X,\|\cdot\|)\).

The main point of the paper is therefore structural.  Earlier work showed that
minimal pairs in the Isbell-convex hull are \(W\)-convex when viewed as
function pairs on the original space \(X\).  In contrast, the present paper
places convexity directly on the hull itself.  This distinction allows one to
study segments, uniqueness, double closure, and fixed point notion inside
\[
(\E(X,\|\cdot\|),q_{\E},\mathbb W),
\]
rather than only through the behaviour of hull elements as functions on \(X\).

Several questions remain open.

\paragraph{Uniqueness of the lifted structure.}
The segment theory developed above shows that uniqueness of \(\mathbb W\) has
strong geometric consequences.  In particular, under uniqueness, every
nontrivial hull segment is isometrically parametrised by a directed interval.
It is therefore natural to ask for intrinsic conditions on
\((X,\|\cdot\|)\) guaranteeing that \(\mathbb W\) is the unique Takahashi
convexity structure on \((\E(X,\|\cdot\|),q_{\E})\).

\paragraph{Normal structure in the hull.}
The fixed point consequences depend on normal-structure assumptions for
bounded, doubly closed, \(\mathbb W\)-convex subsets of the hull.  A useful
next step would be to characterise those subsets of
\(\E(X,\|\cdot\|)\) that have normal structure.  This would clarify when the
fixed point results are genuinely new and when they reduce to known metric or
quasi-metric fixed point principles.

\paragraph{Barycentric and higher-dimensional convexity.}
The present paper works with binary Takahashi convexity maps.  Since
\(\E(X,\|\cdot\|)\) is a real vector space, one may also consider
higher-dimensional barycentric maps of the form
\[
(f_1,\ldots,f_n;\lambda_1,\ldots,\lambda_n)
\longmapsto
\lambda_1 f_1\oplus\cdots\oplus \lambda_n f_n,
\]
where \(\lambda_i\ge0\) and \(\sum_{i=1}^{n}\lambda_i=1\).  It would be
interesting to determine which higher-dimensional analogues of Takahashi
convexity, property \({\rm(S)}\), and uniqueness hold for these maps.

\paragraph{Relation with directed bicombings.}
The maps
\[
\gamma_{f,g}(\lambda)=\mathbb W(f,g,\lambda)
\]
define canonical directed paths in the Isbell-convex hull.  This suggests a
possible connection with bicombing-type structures in metric geometry.  A
future direction is to develop an asymmetric version of conical or convex
bicombings using the lifted Takahashi maps on Isbell-convex hulls.

\paragraph{Stability under hull constructions.}
Another natural question concerns functoriality.  Given a nonexpansive or
linear contraction
\[
T:(X,\|\cdot\|_X)\longrightarrow (Y,\|\cdot\|_Y),
\]
one may ask whether \(T\) induces a map between the corresponding
Isbell-convex hulls that preserves the lifted convexity structures.  Such a
result would strengthen the categorical meaning of the construction and could
lead to a broader theory of convexity-preserving hull extensions.


\end{document}